\documentclass[oneside,british]{amsart}
\usepackage[T1]{fontenc}
\usepackage[latin9]{inputenc}
\usepackage{babel}
\usepackage{amstext}
\usepackage{amsthm}
\usepackage{amssymb}
\usepackage[unicode=true,
 bookmarks=true,bookmarksnumbered=false,bookmarksopen=false,
 breaklinks=false,pdfborder={0 0 1},backref=false,colorlinks=false]
 {hyperref}
\hypersetup{pdftitle={Non-zero constraints in elliptic PDE with random boundary values and applications to hybrid inverse problems},
 pdfauthor={Giovanni S. Alberti}}

\makeatletter
\theoremstyle{definition}
 \newtheorem{example}{\protect\examplename}
\theoremstyle{plain}
\newtheorem{assumption}{\protect\assumptionname}
\theoremstyle{remark}
\newtheorem*{rem*}{\protect\remarkname}
\theoremstyle{plain}
\newtheorem{thm}{\protect\theoremname}
\theoremstyle{plain}
\newtheorem{prop}{\protect\propositionname}
\theoremstyle{plain}
\newtheorem{lem}{\protect\lemmaname}

\usepackage{empheq}
\usepackage[usenames,dvipsnames]{xcolor}
\usepackage{dsfont}

\makeatother

\providecommand{\assumptionname}{Assumption}
\providecommand{\examplename}{Example}
\providecommand{\lemmaname}{Lemma}
\providecommand{\propositionname}{Proposition}
\providecommand{\remarkname}{Remark}
\providecommand{\theoremname}{Theorem}

\begin{document}
\global\long\def\R{\mathbb{R}}%

\global\long\def\ii{\mathrm{i}}%

\global\long\def\N{\mathbb{N}}%

\global\long\def\C{\mathbb{C}}%

\global\long\def\supp{\operatorname{supp}}%

\global\long\def\spn{\operatorname{span}}%

\global\long\def\ran{\operatorname{ran}}%

\global\long\def\rank{\operatorname{rank}}%

\global\long\def\div{\operatorname{div}}%

\global\long\def\supess{\operatorname{supess}}%

\global\long\def\H{\mathcal{H}}%

\global\long\def\k{\omega}%

\global\long\def\phi{\varphi}%

\global\long\def\epsilon{\varepsilon}%

\global\long\def\B{\mathcal{B}}%

\global\long\def\E{\mathbb{E}}%

\global\long\def\D{E}%

\global\long\def\Cp{C_{1}}%

\global\long\def\Cpp{C_{2}}%

\global\long\def\Cppp{C_{3}}%

\global\long\def\P{\mathbb{P}}%

\global\long\def\l{\ell}%

\global\long\def\e{\mathbf{e}}%

\global\long\def\one{\mathds{1}}%

\global\long\def\bo{\partial\Omega}%

\title[Non-zero constraints in elliptic PDE with random boundary values]{Non-zero constraints in elliptic PDE with random boundary values and
applications to hybrid inverse problems}
\author{Giovanni S. Alberti}
\address{MaLGa center, Department of Mathematics, University of Genoa, Via
Dodecaneso 35, 16146 Genova, Italy}
\email{giovanni.alberti@unige.it}
\thanks{This material is based upon work supported by the Air Force Office
of Scientific Research under award number FA8655-20-1-7027. The author
is member of the ``Gruppo Nazionale per l'Analisi Matematica, la Probabilità
e le loro Applicazioni'' (GNAMPA), of the ``Istituto Nazionale di
Alta Matematica'' (INdAM)}
\date{20th September 2022}
\begin{abstract}
Hybrid inverse problems are based on the interplay of two types of
waves, in order to allow for imaging with both high resolution and
high contrast. The inversion procedure often consists of two steps:
first, internal measurements involving the unknown parameters and
some related quantities are obtained, and, second, the unknown parameters
have to be reconstructed from the internal data. The reconstruction
in the second step requires the solutions of certain PDE to satisfy
some non-zero constraints, such as the absence of nodal or critical
points, or a non-vanishing Jacobian.

In this work, we consider a second-order elliptic PDE and show that
it is possible to satisfy these constraints with overwhelming probability
by choosing the boundary values randomly, following a sub-Gaussian
distribution. The proof is based on a new quantitative estimate for
the Runge approximation, a result of independent interest.
\end{abstract}

\maketitle

\section{Introduction}

Many inverse problems for partial differential equations (PDE) are
severely ill-posed, meaning that the best possible stability is of
logarithmic type \cite{alessandrini-1988-stability,MANDACHE-01}.
This phenomenon is typically due do the infinite smoothing effect
of the PDE involved, and has a big impact in the implementation, because
only low-resolution reconstructions are possible, despite the fact
that the parameters under investigation exhibit high contrast and
are relevant for the applications \cite{WIDLAK-SCHERZER-2012,WANG-ANASTASIO-HANDBOOK-2011}.
Several examples include the Calderón problem for electrical impedance
tomography (EIT) \cite{CALDERON-1980,UHLMANNIP-2009}, inverse scattering
problems \cite{COLTON-KRESS-98} and optical tomography \cite{1999-arridge}.
On the other hand, other modalities, such as ultrasonography (modelled
by the wave equation, which preserves singularities) and magnetic
resonance imaging (MRI) exhibit high resolution, but may have low
contrast in certain settings.

In order to overcome this limitation, a class of inverse problems
has been extensively studied over the last decade. These are called
hybrid or coupled-physics problems, because they involve the interplay
of two physical modalities, one providing high-contrast, and one providing
high resolution \cite{kuchment-2012,BAL-2012,2017-ammari-etal,Alberti2018}.
The most known hybrid modality is photoacoustic tomography (PAT) \cite{KUCHMENT-KUNYANSKY-HANBOOK-2011,WANG-ANASTASIO-HANDBOOK-2011},
in which light and ultrasounds are combined to image the high-contrast
optical absorption by making high-resolution ultrasonic measurements.
Many other techniques have been considered, including thermoacoustic
tomography (combining microwaves and ultrasounds) \cite{KUCHMENT-KUNYANSKY-HANBOOK-2011,WANG-ANASTASIO-HANDBOOK-2011},
acousto-electric tomography (combining ultrasound-induced deformations
and electrical measurements) \cite{LAVANDIER-JOSSINET-GATHIGNOL-00,ZHANG-WANG-2004,AMMARI-BONNETIER-CAPDEBOSCQ-TANTER-FINK-2008,adesokan-etal-2018}
and magnetic resonance electrical impedance tomography (combining
MRI and EIT) \cite{seo-woo-2011}.

Because of the combination of two modalities, the inversion process
in hybrid problems usually consists of two steps. In a first step,
by solving an inverse problem for the high-resolution/low-contrast
modality, some internal data is reconstructed. This typically contains
the unknown parameter(s), as well as the solution(s) of the corresponding
PDE related to the low-resolution/high-contrast modality (for instance,
in PAT, the internal data is the product between the optical absorption
$\mu(x)$ and the light intensity $u(x)$). In the second step, the
actual unknown of the problem ($\mu(x)$, in PAT) has to be reconstructed
from the internal data, by solving an inverse problem for the low-resolution/high-contrast
modality. While a lot of research has been done on the first step,
from both the theoretical and experimental points of view, the second
step has received far less attention.

In this paper, we study a crucial aspect for solving the inverse problem
in the second step: the solutions to the PDE under consideration should
satisfy certain non-zero constraints, depending on the problem, such
as the absence of nodal or critical points, or a non-vanishing Jacobian.
It is worth observing that, by using unique continuation estimates,
it is possible to have uniqueness and stability of the inverse problem
even when these non-zero constraints are not satisfied \cite{2014-alessandrini,2015-alessandrini-dicristo-francini-vessella,2015-choulli-triki,2019-choulli-triki,2022-bonnettier-etal}.
However, the non-zero constraints allow for optimal stability estimates
(of Lipschitz type) and, very often, explicit reconstruction formulae
\cite{BAL-2012,Alberti2018}. In very simplified terms (with the risk
of not capturing the complete picture), when these non-zero constraints
are satisfied, it is possible to avoid the problem of ``division by
zero''.

There exist several methods to construct suitable solutions that satisfy
the relevant non-zero constraints, including methods based on generalisations
of the Radó-Kneser-Choquet theorem \cite{alessandrini-1986,ALESSANDRINI-NESI-01,Bauman-Marini-Nesi-2001,ALESSANDRINI-NESI-2015},
on complex geometrical optics solutions (CGO) \cite{BAL-2012,BAL-UHLMANN-2010,BAL-REN-2011,BAL-UHLMANN-2013,BAL-BONNETIER-MONARD-TRIKI-2013},
on the use of multiple frequencies \cite{ALBERTI-IP-2013,albertigsII,ALBERTI-b-2014,alberti-genericity},
on dynamical systems \cite{2013-bal-courdurier} and on the Runge
approximation \cite{BAL-UHLMANN-2013,2022-alberti-capdeboscq}. All
these techniques have several drawbacks, for example the generalisations
of the Radó-Kneser-Choquet theorem are valid only in 2D and for coercive
elliptic problems \cite{CAPDEBOSCQ-15,alberti-bal-dicristo-2016,Alberti2018},
the construction with CGO depends on the unknown coefficients, which
have to be very smooth and isotropic, and the use of multiple frequencies
works only with frequency-dependent problems.

The approach based on the Runge approximation \cite{LAX-1956,MALGRANGE-1955-56}
is very flexible, because it is valid with many PDE, allows for anisotropic
coefficients, and the smoothness assumptions are very mild, since
only the unique continuation property is needed. However, in the original
formulation \cite{BAL-UHLMANN-2013}, the suitable solutions are not
explicitly constructed (the existence follows from the Hahn-Banach
theorem, hence from the axiom of choice) and depend on the unknown
coefficients, as with CGO. By combining this approach with the Whitney
embedding theorem, it is possible to prove that the set of suitable
solutions is open and dense \cite{2022-alberti-capdeboscq}, with
explicit estimates on the number of solutions needed (typically very
small). However, the ``open and dense'' condition does not imply anything
on the ``size'' of this set, since an open and dense set may have
arbitrarily small measure, which may make it hard to construct, in
practice, these solutions. Furthermore, no quantitative lower bound
for the constraints is given, potentially causing the problem of ``division
by a very small number'' (again, by oversimplifying the matter).

The focus of this work is on overcoming these limitations. We consider
a second-order elliptic PDE and prove that, by choosing random boundary
values, sampled with respect to a fixed sub-Gaussian distribution,
the corresponding solutions will satisfy the non-zero constraints
with overwhelming probability (see Theorem~\ref{thm:constraints}).
In this way, the suitable boundary conditions are explicitly constructed,
and independent of the unknown parameters. Furthermore, a quantitative
lower bound is derived. The price to pay is a slightly larger number
of solutions needed if compared to the one in \cite{2022-alberti-capdeboscq},
which is due to the use of concentration inequalities.

A crucial ingredient of the proof of Theorem~\ref{thm:constraints}
is a new result on quantitative Runge approximation. Classical Runge
approximation allows for approximating a local solutions $h$ in $D$
to a PDE by a global solution $u$ in $\Omega$ (with $D\Subset\Omega$)
\cite{LAX-1956,MALGRANGE-1955-56,1962-browder}, and is based on the
unique continuation property \cite{alessandrini-rondi-rosset-vessella-2009}.
This construction was made quantitative in \cite{ruland-salo-2019},
where the authors obtain an estimate on $\|u|_{\bo}\|_{H^{1/2}(\bo)}$.
In the second main result of this paper, Theorem~\ref{thm:quantitative runge},
we obtain an estimate of a stronger norm of $u|_{\bo}$, which is
needed in the proof of Theorem~\ref{thm:constraints}. This result
is also of independent interest, since the Runge approximation finds
applications not only to hybrid problems, but also to many other inverse
problems for PDE (see \cite{ruland-salo-2019} and references therein).

This work is structured as follows. In Section~\ref{sec:non-zero}
we describe the setup we consider, and state the main result of this
work on satisfying non-zero constraints in elliptic PDE with random
boundary measurements, Theorem~\ref{thm:constraints}. In Section~\ref{sec:Quantitative-Runge-approximation}
we state the result on quantitative Runge approximation, Theorem~\ref{thm:quantitative runge}.
These two theorems are then proved in Sections~\ref{sec:Proof-of-Theorem contraint}
and \ref{sec:Proof-of-Runge}, respectively.

\section{Non-zero constraints in PDE with random boundary values\label{sec:non-zero}}

\subsection{The elliptic equation\label{subsec:The-elliptic-equation}}

In this paper, we consider a second-order elliptic equation of the
form
\begin{equation}
Lu:=-\div(a\nabla u)+qu=0\quad\text{in }\Omega,\label{eq:PDE}
\end{equation}
where we assume that 
\begin{itemize}
\item the domain $\Omega\subseteq\R^{d}$, $d\ge2$, is bounded, Lipschitz
and connected;
\item the diffusion coefficient $a\in W^{1,\infty}\left(\Omega;\mathbb{R}^{d\times d}\right)$,
is symmetric ($a^{T}=a$) and satisfies
\begin{equation}
a\left(x\right)\xi\cdot\xi\geq\Lambda^{-1}\left|\xi\right|^{2},\qquad x\in\Omega,\,\xi\in\R^{d},\label{eq:coercive}
\end{equation}
and
\begin{equation}
\|a\|_{W^{1,\infty}(\Omega)}\le\Lambda\label{eq:a-Lambda}
\end{equation}
for some $\Lambda\ge1$;
\item the potential $q\in L^{\infty}(\Omega)$ satisfies
\begin{equation}
\|q\|_{\infty}\le\Lambda;\label{eq:q-Lambda}
\end{equation}
\item $0$ is not a Dirichlet eigenvalue of $L$ in $\Omega$ and
\begin{equation}
\|L^{-1}\|_{H^{-1}(\Omega)\to H_{0}^{1}(\Omega)}\le\Lambda,\label{eq:0-eig}
\end{equation}
where $H^{-1}(\Omega)$ denotes the dual of $H_{0}^{1}(\Omega)=\{u\in H^{1}(\Omega):u=0\text{ on }\partial\Omega$\}.
\end{itemize}
Unless specified, all function spaces considered in this paper consist
of real-valued functions.

\subsection{Non-zero constraints}

Let $\Omega'\Subset\Omega$ be a Lipschitz domain where we are interested
in satisfying the non-zero constraints. The results presented here
would work also with $\Omega'=\Omega$, but with additional technicalities.
We are looking for boundary conditions $\phi_{1},\dots,\phi_{n}\in H^{1/2}(\bo)$
such that the corresponding solutions $u_{1},\dots,u_{n}\in H^{1}(\Omega)$
to
\begin{equation}
\left\{ \begin{array}{ll}
Lu_{i}=-\div(a\nabla u_{i})+qu_{i}=0 & \text{in}\;\Omega,\\
u_{i}=\varphi_{i} & \text{on}\;\bo,
\end{array}\right.\label{eq:ui dirichelt}
\end{equation}
satisfy the constraint
\begin{equation}
\zeta(u_{1},\dots,u_{n})(x)\neq0,\label{eq:constraint}
\end{equation}
in $\Omega'$, locally or globally. Here, $n\ge1$ and the map
\[
\zeta\colon C^{1,1/2}(\overline{\Omega'})^{n}\longrightarrow C^{0,1/2}(\overline{\Omega'})
\]
is multilinear and bounded. Furthermore, we require $\zeta$ to be
local, in the sense that for every $x\in\overline{\Omega'}$ and for
every $u,v\in C^{1,1/2}(\overline{\Omega'})^{n}$ there exists $r>0$
such that, if $u|_{B(x,r)\cap\overline{\Omega'}}=v|_{B(x,r)\cap\overline{\Omega'}}$,
then $\zeta(u)(x)=\zeta(v)(x)$. 

Note that, thanks to (\ref{eq:a-Lambda}), $u_{i}\in C^{1,1/2}(\overline{\Omega'})$
by classical elliptic regularity theory \cite[Theorem 8.32]{GILBARG-2001},
and so (\ref{eq:constraint}) is well-defined. We could have considered
$C^{1,\alpha}(\overline{\Omega'})$ with any $\alpha\in(0,1),$ and
we chose $\alpha=1/2$ only for simplicity. We decided to consider
only Dirichlet boundary values in order to simplify the exposition,
but the whole approach would easily extend to other types of boundary
conditions.

\subsection{Examples\label{subsec:Examples}}

As examples of this general framework, we consider here some particular
cases of maps $\zeta$, and we mention the corresponding applications
to the reconstruction procedures for several hybrid inverse problems.
We present here only simplified settings, we omit the details and
provide only local arguments, in order to better illustrate the role
of the non-zero constraints with toy models. For more complete models
and the full derivations see, e.g., \cite{bal-uhlmann-reconstructions-2012,BAL-2012,kuchment-2012,BAL-UHLMANN-2013,ALBERTI-b-2014,2017-ammari-etal,Alberti2018}
and the references therein.
\begin{example}[Nodal points]
\label{exa:nodal}Let $n=1$ and $\zeta_{1}(u)=u$. This simply yields
the constraint 
\begin{equation}
u(x)\neq0,\label{eq:nodal}
\end{equation}
namely, the (local or global) absence of nodal points. This constraint
appears, for example, in dynamic elastography \cite{MCLAUGHLIN-OBERAI-YOON-2012}
and in quantitative photoacoustic tomography (QPAT) \cite{KUCHMENT-KUNYANSKY-HANBOOK-2011,2012-bal-ren,2015-alberti-ammari}.

For the sake of illustration, let us consider a simplified model of
QPAT in which the absorption coefficient $\mu\in L^{\infty}(\Omega)$,
appearing as a potential in the diffusion equation
\begin{equation}
-\Delta u+\mu u=0\quad\text{in }\Omega,\label{eq:diffusion}
\end{equation}
has to be reconstructed from the knowledge of the internal energy
\[
H=\mu u\quad\text{in }\Omega.
\]
If the boundary condition of $u$ is known, by solving the Poisson
equation
\[
\Delta u=H\quad\text{in }\Omega,
\]
we can recover $u$ in $\Omega$. Finally, $\mu$ can be recovered
in $x\in\Omega$ by using
\begin{equation}
\mu(x)=\frac{H(x)}{u(x)},\label{eq:division}
\end{equation}
provided that the constraint (\ref{eq:nodal}) is satisfied.
\end{example}
\begin{example}[Critical points]
Let $n=1$ and $\zeta_{2}(u)=\partial_{x_{1}}u$. This choice for
$\zeta$ corresponds to the constraint
\[
\partial_{x_{1}}u(x)\neq0,
\]
which implies the absence of critical points. We consider the problem
of reconstructing $a$ in
\[
-\div(a\nabla u)=0\quad\text{in }\Omega,
\]
from the knowledge of the potential $u$ in $\Omega$, as in \cite{alessandrini-1986,giordano-nickl-2020}.
This problem is motivated, for instance, by the inverse problem of
aquifer hydrology \cite{neuman-yakowitz-1979}, but is very related
to the inverse problems of current density imaging \cite{WOO-LEE-SY-MUN-1994,bal-guo-monard-2014}
and of acousto-electric tomography \cite{AMMARI-BONNETIER-CAPDEBOSCQ-TANTER-FINK-2008,adesokan-etal-2018,CAPDEBOSCQ-FEHRENBACH-DEGOURNAY-KAVIAN-09}.
If $a$ is scalar, we obtain
\begin{equation}
\nabla(\log a)\cdot\nabla u=-\Delta u\quad\text{in }\Omega.\label{eq:log a}
\end{equation}
Since $\nabla u$ and $\Delta u$ are known, this is a first-order
linear PDE in $\log a$, and can be solved by using the method of
characteristics if $u$ does not have critical points. 
\end{example}
\begin{example}[Jacobian]
\label{exa:Jacobian}Let $n=d$ and $\zeta_{3}(u_{1},\dots,u_{d})=\det\begin{bmatrix}\nabla u_{1} & \cdots & \nabla u_{d}\end{bmatrix}$,
which yields the constraint
\begin{equation}
\det\begin{bmatrix}\nabla u_{1} & \cdots & \nabla u_{d}\end{bmatrix}(x)\neq0.\label{eq:jacobian}
\end{equation}
In other words, we look for a non-vanishing Jacobian. A simple application
of this constraint is in the same inverse problem considered in the
previous example. Suppose that we have $d$ measurements $u_{1},\dots,u_{d}$.
By (\ref{eq:log a}) we obtain
\[
\left(\nabla(\log a)\right)^{T}\begin{bmatrix}\nabla u_{1} & \cdots & \nabla u_{d}\end{bmatrix}=-\begin{bmatrix}\Delta u_{1} & \cdots & \Delta u_{d}\end{bmatrix}\quad\text{in }\Omega.
\]
If (\ref{eq:jacobian}) holds true, we can immediately obtain $\nabla(\log a)$,
hence $a$ up to a multiplicative constant, by inverting the matrix
$\begin{bmatrix}\nabla u_{1} & \cdots & \nabla u_{d}\end{bmatrix}$
in the above identity.
\end{example}
\begin{example}[Augmented Jacobian]
Let $n=d+1$ and 
\[
\zeta_{4}(u_{1},\dots,u_{d+1})=\det\begin{bmatrix}u_{1} & \cdots & u_{d+1}\\
\nabla u_{1} & \cdots & \nabla u_{d+1}
\end{bmatrix}.
\]
This yields the constraint
\begin{equation}
\det\begin{bmatrix}u_{1} & \cdots & u_{d+1}\\
\nabla u_{1} & \cdots & \nabla u_{d+1}
\end{bmatrix}(x)\neq0,\label{eq:augmented}
\end{equation}
the so-called non-vanishing \emph{augmented} Jacobian. This constraint
appears, for instance, in QPAT. As in Example~\ref{exa:nodal}, consider
again the inverse problem of recovering the absorption coefficient
$\mu$ in (\ref{eq:diffusion}) from the internal energy
\[
H(x)=\Gamma(x)\mu(x)u(x),
\]
where $\Gamma$ is the Grüneisen parameter, which is now supposed
unknown. Suppose we have at our disposal $d+1$ measurements $H_{i}=\Gamma\mu u_{i}$
with $u_{1}\neq0$ in $\Omega$ (see the constraint in Example~\ref{exa:nodal}).
Setting $v_{i}=u_{i}/u_{1}$, it is easy to show that
\[
-\div(u_{1}^{2}\nabla v_{i})=0\quad\text{in }\Omega.
\]
Further, (\ref{eq:augmented}) implies that
\[
\det\begin{bmatrix}\nabla v_{2} & \cdots & \nabla v_{d+1}\end{bmatrix}(x)\neq0,
\]
see \cite{Alberti2018,AMMARI-CAPDEBOSCQ-DEGOURNAY-ROZANOVA-TRIKI-2011}.
Thus, since $v_{i}=H_{i}/H_{1}$ is known, $u_{1}^{2}$ may be recovered
by arguing as in Example~\ref{exa:Jacobian}. Once $u_{1}$ is known,
$\mu$ may be recovered as in Example~\ref{exa:nodal}, see (\ref{eq:division}).
\end{example}

\subsection{Main result}

Before giving our main result, we need to state two hypotheses. In
the first one, we need to assume that, at least in the case of constant
coefficients, the constraint (\ref{eq:constraint}) can be satisfied.
\begin{assumption}
\label{assu:runge-zeta}Let $\D>0$. We assume that for every $x_{0}\in\overline{\Omega'}$
there exist $u_{1}^{0},\dots,u_{n}^{0}\in C^{1,1/2}(\overline{\Omega})$
such that $\div(a(x_{0})\nabla u_{i}^{0})=0$ in $\Omega$, $\|u_{i}^{0}\|_{C^{1,1/2}(\overline{\Omega})}\le\D$
for every $i=1,\dots,n$ and 
\begin{equation}
|\zeta(u_{1}^{0},\dots,u_{n}^{0})(x_{0})|\ge1.\label{eq:assu-runge}
\end{equation}
\end{assumption}
\begin{rem*}
We explicitly observe that this assumption is trivially satisfied
for all $a$ and for all the constraints introduced in $\S$\ref{subsec:Examples}.
Indeed, the functions $1$, $x_{1},\dots,x_{d}$ trivially satisfy
the PDE $\div(a(x_{0})\nabla u_{i}^{0})=0$ (because their gradients
are constant), and
\[
\zeta_{1}(1)=\zeta_{2}(x_{1})=\zeta_{3}(x_{1},\dots,x_{d})=\zeta_{4}(1,x_{1},\dots,x_{d})=1,
\]
and so (\ref{eq:assu-runge}) is trivially satisfied.
\end{rem*}
The key aspect of the approach presented in this paper lies in the
random choice for the boundary conditions in (\ref{eq:ui dirichelt})
in order to satisfy (\ref{eq:constraint}), with respect to the following
distribution.
\begin{assumption}
\label{assu:phi}Let $\varphi$ be a square-integrable random variable
in $H^{\frac{1}{2}}(\bo)$, and let $\nu$ denote the corresponding
distribution, so that $\varphi\sim\nu$. We assume that $\phi$ is
sub-Gaussian, has mean $0$ and that its covariance operator $\Sigma\colon H^{\frac{1}{2}}(\bo)\to H^{\frac{1}{2}}(\bo)$
is injective.
\end{assumption}
\begin{rem*}
Let us make this assumption more precise. Given a probability space
$(X,\mathcal{F},\mu)$, a random variable in $H^{\frac{1}{2}}(\bo)$
is a measurable map $\phi\colon X\to H^{\frac{1}{2}}(\bo)$. We use
$\phi$ to push-forward the measure $\mu$ on $X$ to a measure $\nu$
on $H^{\frac{1}{2}}(\bo)$. The covariance operator is defined by
$\Sigma=\E[\phi\otimes\phi]$, and is self-adjoint, positive and trace-class.

We follow \cite{fukuda-1990,2020-giorgobani-etal} for defining a
sub-Gaussian random variable in a Hilbert space. We say that $\phi$
is sub-Gaussian if there exists $C>0$ such that
\[
\tau\left(\langle\phi,\psi\rangle_{H^{\frac{1}{2}}(\bo)}\right)\le C\left(\E|\langle\phi,\psi\rangle_{H^{\frac{1}{2}}(\bo)}|^{2}\right)^{\frac{1}{2}},\qquad\psi\in H^{\frac{1}{2}}(\bo),
\]
where, for a real random variable $\xi,$ we define
\[
\tau(\xi)=\inf\{a\ge0:\E\,e^{t\xi}\le e^{\frac{t^{2}}{2}a^{2}}\;\text{for every}\;t\in\R\}.
\]
\end{rem*}
The simplest choice for $\phi$ is any non-degenerate Gaussian random
variable in $H^{1/2}(\bo)$. More explicitly, $\phi$ may be expressed
as
\begin{equation}
\varphi=\sum_{k\in\N}a_{k}e_{k},\label{eq:phi Gaussian}
\end{equation}
where $\{e_{k}\}_{k}$ is a fixed orthonormal basis of $H^{\frac{1}{2}}(\bo)$
and $a^{k}\sim N(0,\sigma_{k}^{2})$ are independent real Gaussian
variables, with $\sigma_{k}>0$ for every $k$ (since $\Sigma$ is
injective) and $\sum_{k}\sigma_{k}<+\infty$ (since $\Sigma$ is a
trace-class operator). Decompositions of the form (\ref{eq:phi Gaussian}),
which will appear several times in the sequel, are called Karhunen--Loève
expansions \cite{2019-steinwart}, and converge in norm.

The main result of this section reads as follows.
\begin{thm}
\label{thm:constraints}There exist $\Cp,\Cpp,\Cppp>0$ depending
only on $\Omega$, $\Omega'$, $\Lambda$, $\D$, $\zeta$ and $\nu$
such that the following is true. Take $N\in\N,$ $N\ge n^{\frac{1}{d-1}}$,
and let $\phi_{i}^{l}\sim\nu$ be sampled i.i.d.\ in $H^{\frac{1}{2}}(\bo)$
for $i=1,\dots,n$ and $l=1,\dots,N$. Then, with probability greater
than
\[
1-\Cp N^{d}\exp(-\Cpp N^{\frac{1}{n}})
\]
we have that
\begin{equation}
\max_{l=1,\dots,N}|\zeta(u_{1}^{l},\dots,u_{n}^{l})(x)|\ge\Cppp,\qquad x\in\overline{\Omega'},\label{eq:max}
\end{equation}
where $u_{i}^{l}$ is the solution to (\ref{eq:ui dirichelt}) with
boundary condition $\phi_{i}^{l}$.
\end{thm}
Several comments on this result are in order.
\begin{itemize}
\item The positive constants $\Cp$, $\Cpp$ and $\Cppp$ depend only on
the priori data on the problem and are independent of the unknown
parameters $a$ and $q$ of the PDE.
\item The lower bound for the probability, $1-\Cp N^{d}\exp(-\Cpp N^{\frac{1}{n}})$,
converges exponentially to $1$ as $N$ grows, and so a small number
$N$ of measurements is sufficient to enforce the constraint with
overwhelming probability. Further, the parameters $d$ and $n$ appear
directly in the estimate: the larger they are, the worse, and this
is expected, because $d$ and $n$ can be seen as underlying dimensions
of the problem.
\item Even if the number $N$ of measurements is expected to be small, for
$N>1$ condition (\ref{eq:max}) guarantees that the constraint is
verified only locally in $\overline{\Omega'}$. More precisely, by
(\ref{eq:max}) we have the open cover
\begin{equation}
\overline{\Omega'}\subseteq\bigcup_{l=1}^{N}\Omega_{l},\qquad\Omega_{l}=\{x\in\Omega:|\zeta(u_{1}^{l},\dots,u_{n}^{l})(x)|>\Cppp/2\}.\label{eq:cover Omega;}
\end{equation}
In other words, the domain $\overline{\Omega'}$ can be covered by
$N$ subdomains, and in each of them the constraint (\ref{eq:constraint})
is satisfied for the measurement $l$. Therefore, the local reconstruction
procedures described in $\S$\ref{subsec:Examples} may be carried
out in each subdomain separately, by using the corresponding measurement
$l$.
\item It is worth observing that (\ref{eq:max}), or, equivalently, the
cover (\ref{eq:cover Omega;}), guarantees that (\ref{eq:constraint})
is satisfied with a quantitative lower bound $\Cppp$ depending only
on the a priori data. This lower bound is crucial from the numerical
point of view, because it makes all the inversion steps discussed
in $\S$\ref{subsec:Examples} well conditioned.
\item The fact that (\ref{eq:max}) is true only with overwhelming probability,
and not with probability $1$, is unavoidable with this approach,
where we aim for a quantitative lower bound with randomly-chosen (i.i.d.)
boundary values. Indeed, in general, the probability that all the
boundary conditions are not suitably chosen will be positive. It would
be interesting to investigate whether the random choice of the boundary
values could be modified to improve upon the number of solutions,
possibly in combination with the approach based on the Whitney embedding
theorem introduced in \cite{2022-alberti-capdeboscq}: we leave this
for future work.
\end{itemize}
The proof of this result is presented in Section~\ref{sec:Proof-of-Theorem contraint},
and is based on the two following steps:
\begin{enumerate}
\item Letting $\phi_{i}\sim\nu$ i.i.d.\ for $i=1,\dots,n$, and denoting
the corresponding solutions to (\ref{eq:ui dirichelt}) by $u_{i}$,
we prove a lower bound of the form
\[
\mathbb{E}\left(\zeta(u_{1},\dots,u_{n})(x)^{2}\right)\ge C>0,\qquad x\in\overline{\Omega'}.
\]
\item A concentration inequality and regularity estimates for the $u_{i}$s
allow us to move from an estimate in expectation to an estimate in
probability in the whole domain.
\end{enumerate}
For step (1), we shall need a quantitative version of the Runge approximation
property, which is the main result of the following section.

\section{Quantitative Runge approximation\label{sec:Quantitative-Runge-approximation}}

We consider the PDE introduced in $\S$\ref{subsec:The-elliptic-equation}.
The following result is known as the Runge approximation property
\cite{LAX-1956,MALGRANGE-1955-56}, which establishes that local solutions
to
\[
Lu=-\div(a\nabla u)+qu=0\quad\text{in }\Omega,
\]
 (recall (\ref{eq:PDE})) can always be approximated by global solutions.
\begin{prop}
\label{prop:runge}Let $D\Subset\Omega$ be a Lipschitz domain such
that $\Omega\setminus\overline{D}$ is connected. Take $h\in H^{1}(D)$
such that $Lh=0$ in $D$ and $\epsilon>0$. Then there exists $u\in H^{1}(\Omega)$
such that $Lu=0$ in $\Omega$ and
\[
\|h-u|_{D}\|_{L^{2}(D)}\le\epsilon.
\]
\end{prop}
In the proof of Theorem~\ref{thm:constraints}, we will need a quantitative
version of this result, where a suitable norm of $u|_{\bo}$ is estimated.
This will be the content of the main result of this section.

We need to introduce the class of Lipschitz subdomains $D$ we are
going to consider. Following \cite[Chapter VI,  Section 3]{stein-1970},
we say that an open set $D\subseteq\R^{d}$ is $\Lambda$-Lipschitz,
for a Lipschitz bound $\Lambda\ge1$, if there exists a family of
open sets $\{U_{i}\}_{i}$ such that:
\begin{enumerate}
\item If $x\in\partial D$, then $B(x,\Lambda^{-1})\subseteq U_{i}$ for
some $i$, where $B(x,r)$ denotes the open ball in $\R^{d}$ of centre
$x$ and radius $r$, namely, $B(x,r)=\{y\in\R^{d}:|y-x|<r\}$;
\item The intersection of more than $\left\lfloor \Lambda\right\rfloor $
of the sets $U_{i}$ is always empty, where $\left\lfloor \Lambda\right\rfloor $
is the integer part of $\Lambda$;
\item For every $i$, up to a rotation of the axes, $D\cap U_{i}=D_{i}\cap U_{i}$,
where $D_{i}$ is the set of points lying above a hypersurface with
Lipschitz graph, with Lipschitz constant smaller than or equal to
$\Lambda$.
\end{enumerate}
We are ready to state the main result on quantitative Runge approximation.
\begin{thm}
\label{thm:quantitative runge}Let $\Omega\subseteq\R^{d}$, $d\ge2$,
be a bounded, Lipschitz and connected domain, $\Lambda\ge1$, $\epsilon>0$
and $f\colon H^{1/2}(\bo)\to[0,+\infty]$ be positively homogeneous
and such that $f^{-1}([0,+\infty))$ is dense in $H^{1/2}(\bo)$.
Then, there exists $C>0$ depending only on $\Omega$, $\Lambda$,
$\epsilon$ and $f$ such that the following is true.

Let $a\in W^{1,\infty}\left(\Omega;\mathbb{R}^{d\times d}\right)$
and $q\in L^{\infty}(\Omega)$ satisfy $a^{T}=a$ and (\ref{eq:coercive}),
(\ref{eq:a-Lambda}), (\ref{eq:q-Lambda}) and (\ref{eq:0-eig}).
Let $D\Subset\Omega$ be a convex domain such that
\begin{equation}
|D|\ge\Lambda^{-1},\qquad d(\overline{D},\bo)\ge\Lambda^{-1},\qquad\text{\ensuremath{D} is \ensuremath{\Lambda}-Lipschitz.}\label{eq:hypotheses D}
\end{equation}
For every $h\in H^{1}(D)$ such that $Lh=0$ in $D$ there exists
$u\in H^{1}(\Omega)$ such that $Lu=0$ in $\Omega$ and
\begin{equation}
\|h-u|_{D}\|_{L^{2}(D)}\le\epsilon\|h\|_{H^{1}(D)},\qquad f(u|_{\bo})\le C\|h\|_{H^{1}(D)}.\label{eq:thesis}
\end{equation}
\end{thm}
Several comments on this result are in order:
\begin{itemize}
\item As it will be clear from the proof, the assumption on the convexity
of $D$ may be removed, provided that $\Omega\setminus\overline{D}$
is connected. However, in this case, the constant $C$ will depend
also on $D$, and not only on its Lipschitz bound $\Lambda$.
\item The function $f$ may be chosen as a stronger (semi-)norm on $\partial\Omega$,
such as 
\begin{equation}
f(\phi)=\|\phi\|_{H^{s}(\bo)}\label{eq:Hs}
\end{equation}
 for any $s\ge\frac{1}{2}$, because $f(z\phi)=|z|f(\phi)$ for all
$z\in\R$. Further, the density of $f^{-1}([0,+\infty))=H^{s}(\bo)$
in $H^{1/2}(\bo)$ follows from the density of smooth functions in
$H^{1/2}(\bo)$.
\item For the proof of Theorem~\ref{thm:constraints}, we will apply this
result with
\begin{equation}
f(\phi)=\left(\sum_{k\in\N}\frac{|a_{k}|^{2}}{\sigma_{k}^{2}}\right)^{\frac{1}{2}}\in[0,+\infty],\qquad\phi=\sum_{k\in\N}a_{k}e_{k}\in H^{\frac{1}{2}}(\bo),\label{eq:spoiler}
\end{equation}
where we use the notation of (\ref{eq:phi Gaussian}). It is worth
observing that, if $\Omega=B(0,1)\subseteq\R^{2}$ and the basis $\{e_{k}\}$
is the Fourier basis on $\partial B(0,1)$, this map $f$, for a suitable
choice of weights $\sigma_{k}$, corresponds to the one in (\ref{eq:Hs}).
Thus, in some sense, $f(\phi)$ may be seen as a generalised Sobolev
norm on $\bo$.
\item If we compare this result with the quantitative version of the Runge
approximation property derived in \cite{ruland-salo-2019}, we observe
that in \cite{ruland-salo-2019} the second bound in (\ref{eq:thesis})
is replaced by
\[
\|u|_{\bo}\|_{H^{1/2}(\bo)}\le Ce^{C\epsilon^{-\mu}}\|h\|_{L^{2}(D)},
\]
where $C$ and $\mu$ are positive constants that are independent
of $\epsilon$. Thus, the estimate in Theorem~\ref{thm:quantitative runge}
allows for controlling more general norms of $u|_{\bo}$ (and this
will be needed in the proof of Theorem~\ref{thm:constraints}), but
is weaker because no explicit dependence on $\epsilon$ is given.
This is due to the fact that the proof of Theorem~\ref{thm:quantitative runge}
is based on an argument by contradiction and is not constructive,
in contrast to the argument of \cite{ruland-salo-2019}, which is
based on quantitative unique continuation estimates. Another difference
lies in the fact that in the results of \cite{ruland-salo-2019},
the constant $C$ depends on $D$, while here it depends only on its
Lipschitz bound $\Lambda$.
\end{itemize}

\section{\label{sec:Proof-of-Theorem contraint}Proof of Theorem~\ref{thm:constraints}}

Let us now express the random variable $\phi$ more explicitly. By
the Karhunen--Loève theorem \cite{2019-steinwart}, there exists
an orthonormal basis $\{e_{k}:k\in\N\}$ of $H^{\frac{1}{2}}(\bo)$
such that
\begin{equation}
\varphi=\sum_{k\in\N}a^{k}e_{k},\label{eq:phi_i}
\end{equation}
where $\{a^{k}\}_{k}$ are real random variables, have zero-mean,
are pairwise uncorrelated and have variance $\sigma_{k}^{2}$. More
precisely, we have
\begin{equation}
\E(a^{k})=0,\qquad\E(a^{k}a^{k'})=\delta_{kk'}\sigma_{k}^{2},\qquad k,k'\in\N.\label{eq:uncorrelated}
\end{equation}
Note that $\sigma_{k}>0$ for every $k$, because the covariance of
$\phi$ is injective (Assumption~\ref{assu:phi}). It is worth observing
that the example of the Gaussian random variable presented in (\ref{eq:phi Gaussian})
is a particular case of this construction, with $a^{k}$ independent
Gaussian random variables.

As in the statement of Theorem~\ref{thm:constraints}, let $\phi_{i}\sim\nu$
be sampled i.i.d.\ in $H^{\frac{1}{2}}(\bo)$ for $i=1,\dots,n$,
and let $u_{i}$ be the corresponding solution to (\ref{eq:ui dirichelt}):
\begin{equation}
\tag{{\ref{eq:ui dirichelt}}}\left\{ \begin{array}{ll}
Lu_{i}:=-\div(a\nabla u_{i})+qu_{i}=0 & \text{in}\;\Omega,\\
u_{i}=\varphi_{i} & \text{on}\;\bo.
\end{array}\right.\label{eq:ui dirichelt-1}
\end{equation}
Note that $u_{i}$ is itself a random variable in $H^{1}(\Omega)$,
and consequently $\zeta(u_{1},\dots,u_{n})$ is a random variable
in $C^{0,1/2}(\overline{\Omega'})$. The mean of $\zeta(u_{1},\dots,u_{n})$
is zero, because $\phi_{i}$ has mean zero, $L$ is linear and $\zeta$
is multilinear. Its variance is computed in the following lemma.
\begin{lem}
\label{lem:expectaction}We have
\[
\E\left(\zeta(u_{1},\dots,u_{n})^{2}\right)=\sum_{k_{1},\dots,k_{n}\in\N}\sigma_{k_{1}}^{2}\cdot\cdots\cdot\sigma_{k_{n}}^{2}\zeta(z_{k_{1}},\dots,z_{k_{n}})^{2},
\]
where $z_{k}$ is the unique solution to
\[
\left\{ \begin{array}{ll}
-\div(a\nabla z_{k})+qz_{k}=0 & \text{in}\;\Omega,\\
z_{k}=e_{k} & \text{on}\;\bo.
\end{array}\right.
\]
\end{lem}
\begin{proof}
In view of (\ref{eq:phi_i}), we can write $\varphi_{i}=\sum_{k\in\N}a_{i}^{k}e_{k}$.
Thus, using the fact that the PDE (\ref{eq:ui dirichelt}) is linear,
we have $u_{i}=\sum_{k}a_{i}^{k}z_{k}$. Thus, since $\zeta$ is multilinear
and bounded, we have
\[
\begin{split}\zeta(u_{1},\dots,u_{n}) & =\zeta\left(\sum_{k_{1}\in\N}a_{1}^{k_{1}}z_{k_{1}},\dots,\sum_{k_{n}\in\N}a_{n}^{k_{n}}z_{k_{n}}\right)\\
 & =\sum_{k_{1},\dots,k_{n}\in\N}a_{1}^{k_{1}}\cdots a_{n}^{k_{n}}\zeta(z_{k_{1}},\dots,z_{k_{n}}).
\end{split}
\]
We readily derive
\[
\begin{split} & \zeta(u_{1},\dots,u_{n})^{2}\\
 & =\left(\sum_{k_{1},\dots,k_{n}\in\N}a_{1}^{k_{1}}\cdots a_{n}^{k_{n}}\zeta(z_{k_{1}},\dots,z_{k_{n}})\right)\left(\sum_{k'_{1},\dots,k'_{n}\in\N}a_{1}^{k'_{1}}\cdots a_{n}^{k'_{n}}\zeta(z_{k'_{1}},\dots,z_{k'_{n}})\right)\\
 & =\sum_{k_{1},k_{1}',\dots,k_{n},k_{n}'}\left(a_{1}^{k_{1}}a_{1}^{k'_{1}}\right)\cdots\left(a_{n}^{k_{n}}a_{n}^{k'_{n}}\right)\zeta(z_{k_{1}},\dots,z_{k_{n}})\zeta(z_{k'_{1}},\dots,z_{k'_{n}}).
\end{split}
\]
Note that $z_{k}$ is a deterministic quantity. Taking the expectation,
by using that the random variables $a_{i}^{k}$ and $a_{j}^{l}$ are
independent if $i\neq j$, we obtain
\begin{multline*}
\E\left(\zeta(u_{1},\dots,u_{n})^{2}\right)\\
=\sum_{k_{1},k_{1}',\dots,k_{n},k_{n}'}\E\left(a_{1}^{k_{1}}a_{1}^{k'_{1}}\right)\cdots\E\left(a_{n}^{k_{n}}a_{n}^{k'_{n}}\right)\zeta(z_{k_{1}},\dots,z_{k_{n}})\zeta(z_{k'_{1}},\dots,z_{k'_{n}}).
\end{multline*}
By (\ref{eq:uncorrelated}), we have
\[
\E\left(\zeta(u_{1},\dots,u_{n})^{2}\right)=\sum_{k_{1},\dots,k_{n}}\sigma_{k_{1}}^{2}\cdots\sigma_{k_{n}}^{2}\zeta(z_{k_{1}},\dots,z_{k_{n}})^{2},
\]
as desired.
\end{proof}
The next lemma shows that it is possible to approximate locally a
solution to the PDE $Lu=0$ by a solution to the PDE with constant
coefficients and without the zeroth order term.
\begin{lem}[{\cite[Proposition~7.10]{Alberti2018}}]
\label{lem:prop710}Take $\delta,\D>0$ and $u^{0}\in C^{1,1/2}(\overline{\Omega})$
such that $\|u^{0}\|_{C^{1,1/2}(\overline{\Omega})}\le\D$. There
exists $r>0$ depending only on $\Omega$, $\Omega'$,  $\Lambda$,
$\D$ and $\delta$ such that for every $x_{0}\in\overline{\Omega'}$
such that $\div(a(x_{0})\nabla u^{0})=0$ in $B(x_{0},r)$ there exists
$u\in H^{1}(B(x_{0},r))$ such that $Lu=0$ in $B(x_{0},r)$ and
\[
\|u-u^{0}\|_{C^{1,1/2}(\overline{B(x_{0},r)})}\le\delta.
\]
\end{lem}
As anticipated in (\ref{eq:spoiler}), we now define the following
stronger ``norm'' on $H^{\frac{1}{2}}(\bo)$: 
\begin{equation}
\|\phi\|_{\sigma}^{2}=\sum_{k\in\N}\frac{|a_{k}|^{2}}{\sigma_{k}^{2}}\in[0,+\infty],\qquad\phi=\sum_{k\in\N}a_{k}e_{k}\in H^{\frac{1}{2}}(\bo).\label{eq:norm sigma}
\end{equation}
Note that this is well defined because $\sigma_{k}>0$ for every $k$
and that it is stronger than $\|\cdot\|_{H^{\frac{1}{2}}(\bo)}$ because
$\sigma_{k}\to0$ (since the covariance of the random variable $\phi$
is trace class). The following lemma is a consequence of Assumption~\ref{assu:runge-zeta}
and of Theorem~\ref{thm:quantitative runge}.
\begin{lem}
\label{lem:tilde phi}Take $x_{0}\in\overline{\Omega'}$. There exist
$C>0$ depending only on $\Omega$, $\Omega'$, $\Lambda$, $\D$,
$\zeta$ and $\nu$ and $\tilde{\phi}_{1},\dots,\tilde{\phi}_{n}\in H^{1/2}(\bo)$
such that
\[
\|\tilde{\phi}_{i}\|_{\sigma}\le C,\qquad i=1,\dots,n
\]
and
\[
|\zeta(\tilde{u}_{1},\dots,\tilde{u}_{n})(x_{0})|\ge1/2,
\]
where
\[
\left\{ \begin{array}{ll}
-\div(a\nabla\tilde{u}_{i})+q\tilde{u}_{i}=0 & \text{in}\;\Omega,\\
\tilde{u}_{i}=\tilde{\varphi}_{i} & \text{on}\;\bo.
\end{array}\right.
\]
\end{lem}
\begin{proof}
By an abuse of notation, several positive constants depending only
on $\Omega$, $\Omega'$, $\Lambda$, $\D$, $\zeta$ and $\nu$  will
be denoted by the same letter $C$.

Take $\epsilon\in(0,1]$ to be chosen later. Let $u_{1}^{0},\dots,u_{n}^{0}\in C^{1,1/2}(\overline{\Omega})$
be as in Assumption~\ref{assu:runge-zeta}. By Lemma~\ref{lem:prop710}
there exist $r\in(0,d(\overline{\Omega'},\bo)/2]$ depending only
on $\Omega$, $\Omega'$, $\Lambda$, $\D$ and $\epsilon$ and $u_{1},\dots,u_{n}\in H^{1}(B(x_{0},r))$
such that $Lu_{i}=0$ in $B(x_{0},r)$ and
\begin{equation}
\|u_{i}-u_{i}^{0}\|_{C^{1,1/2}(\overline{B(x_{0},r)})}\le\epsilon,\qquad i=1,\dots,n.\label{eq:localbound}
\end{equation}
In particular, we have
\begin{equation}
\|u_{i}\|_{C^{1,1/2}(\overline{B(x_{0},r)})}\le\D+1,\qquad i=1,\dots,n.\label{eq:uibound}
\end{equation}

Next, note that $f\colon H^{1/2}(\bo)\to[0,+\infty]$ defined by $f(\phi)=\|\phi\|_{\sigma}$
is positively homogeneous and such that $f^{-1}([0,+\infty))$ is
dense in $H^{1/2}(\bo)$, since $\spn\{e_{k}\}_{k}\subseteq f^{-1}([0,+\infty))$.
Therefore, by Theorem~\ref{thm:quantitative runge} with $D=B(x_{0},r)$
there exists $\tilde{u}_{1},\dots,\tilde{u}_{n}\in H^{1}(\Omega)$
such that $L\tilde{u}_{i}=0$ in $\Omega$ and
\begin{align*}
 & \|u_{i}-\tilde{u}_{i}|_{B(x_{0},r)}\|_{L^{2}(B(x_{0},r))}\le\epsilon\|u_{i}\|_{H^{1}(B(x_{0},r))},\\
 & \|\tilde{u}_{i}|_{\bo}\|_{\sigma}\le C(\Omega,\Omega',\Lambda,\nu,\epsilon,r)\|u_{i}\|_{H^{1}(B(x_{0},r))}.
\end{align*}
Thus, by (\ref{eq:uibound}) we have
\[
\|u_{i}-\tilde{u}_{i}|_{B(x_{0},r)}\|_{L^{2}(B(x_{0},r))}\le C(\D)\epsilon,\qquad\|\tilde{u}_{i}|_{\bo}\|_{\sigma}\le C(\Omega,\Omega',\Lambda,\D,\nu,\epsilon).
\]
Classical elliptic regularity \cite[Theorem 8.32]{GILBARG-2001} yields
\[
\|u_{i}-\tilde{u}_{i}|_{B(x_{0},r)}\|_{C^{1,1/2}(\overline{B(x_{0},r/2)})}\le C\|u_{i}-\tilde{u}_{i}|_{B(x_{0},r)}\|_{L^{2}(B(x_{0},r))}\le C\epsilon.
\]
As a consequence, by (\ref{eq:localbound}) we have
\[
\|u_{i}^{0}-\tilde{u}_{i}\|_{C^{1,1/2}(\overline{B(x_{0},r/2)})}\le\epsilon+C\epsilon\le C\epsilon.
\]
Therefore, by the fact that $\zeta$ is continuous and local, in view
of (\ref{eq:assu-runge}) we can choose $\epsilon\in(0,1]$ depending
only on $C$ and $\zeta$ such that 

\[
|\zeta(\tilde{u}_{1},\dots,\tilde{u}_{n})(x_{0})|\ge1/2.
\]
Choosing $\tilde{\phi}_{i}=\tilde{u}_{i}|_{\bo}$ concludes the proof.
\end{proof}
In the next result we prove a lower bound for the variance of $\zeta(u_{1},\dots,u_{n})$.
\begin{lem}
\label{lem:mean eta}We have
\[
\E\left(\zeta(u_{1},\dots,u_{n})(x)^{2}\right)\ge\eta,\qquad x\in\overline{\Omega'}.
\]
for some $\eta\in(0,1]$ depending only on $\Omega$, $\Omega'$,
 $\Lambda$, $\D$, $\zeta$ and $\nu$.
\end{lem}
\begin{proof}
By an abuse of notation, several positive constants depending only
on $\Omega$, $\Omega'$,  $\Lambda$, $\D$, $\zeta$ and $\nu$
will be denoted by the same letter $C>0$.

Take $x_{0}\in\overline{\Omega'}$. Let $\tilde{\phi}_{1},\dots,\tilde{\phi}_{n}\in H^{1/2}(\bo)$
be as in Lemma~\ref{lem:tilde phi}. Thus, in view of (\ref{eq:norm sigma})
we have
\begin{equation}
\|\tilde{\phi}_{i}\|_{\sigma}^{2}=\sum_{k_{i}\in\N}\frac{|\tilde{a}_{i}^{k_{i}}|^{2}}{\sigma_{k_{i}}^{2}}\le C,\qquad\tilde{\phi}_{i}=\sum_{k_{i}\in\N}\tilde{a}_{i}^{k_{i}}e_{k_{i}}.\label{eq:bound tilde phi}
\end{equation}
In particular, we have $\tilde{u}_{i}=\sum_{k_{i}\in\N}\tilde{a}_{i}^{k_{i}}z_{k_{i}}$,
so that
\[
\begin{split}\frac{1}{2} & \le|\zeta(\tilde{u}_{1},\dots,\tilde{u}_{n})(x_{0})|\\
 & =\left|\sum_{k_{1},\dots,k_{n}\in\N}\tilde{a}_{1}^{k_{1}}\cdots\tilde{a}_{n}^{k_{n}}\zeta(z_{k_{1}},\dots,z_{k_{n}})(x_{0})\right|\\
 & =\left|\sum_{k_{1},\dots,k_{n}\in\N}\frac{\tilde{a}_{1}^{k_{1}}\cdots\tilde{a}_{n}^{k_{n}}}{\sigma_{k_{1}}\cdots\sigma_{k_{n}}}(\sigma_{k_{1}}\cdots\sigma_{k_{n}})\zeta(z_{k_{1}},\dots,z_{k_{n}})(x_{0})\right|.
\end{split}
\]
By using the Cauchy-Schwarz inequality we obtain
\[
\frac{1}{2}\le\left(\sum_{k_{1},\dots,k_{n}}\frac{(\tilde{a}_{1}^{k_{1}})^{2}\cdots(\tilde{a}_{n}^{k_{n}})^{2}}{\sigma_{k_{1}}^{2}\cdots\sigma_{k_{n}}^{2}}\right)^{\frac{1}{2}}\left(\sum_{k_{1},\dots,k_{n}}(\sigma_{k_{1}}^{2}\cdots\sigma_{k_{n}}^{2})\zeta(z_{k_{1}},\dots,z_{k_{n}})(x_{0})^{2}\right)^{\frac{1}{2}}.
\]
Observe that, by (\ref{eq:bound tilde phi}) we have
\[
\sum_{k_{1},\dots,k_{n}\in\N}\frac{(\tilde{a}_{1}^{k_{1}})^{2}\cdots(\tilde{a}_{n}^{k_{n}})^{2}}{\sigma_{k_{1}}^{2}\cdots\sigma_{k_{n}}^{2}}=\prod_{i=1}^{n}\sum_{k_{i}\in\N}\frac{(\tilde{a}_{i}^{k_{i}})^{2}}{\sigma_{k_{i}}^{2}}=\prod_{i=1}^{n}\|\tilde{\phi}_{i}\|_{\sigma}^{2}\le C.
\]
As a consequence, by Lemma~\ref{lem:expectaction} we obtain
\[
\frac{1}{2}\le C\E\left(\zeta(u_{1},\dots,u_{n})(x_{0})^{2}\right)^{\frac{1}{2}},
\]
and the result follows.
\end{proof}
We say that a real-valued random variable $X$ is $\gamma$-subexponential
if there exists $K>0$ such that $\E\exp(|X|^{\gamma}/K^{\gamma})\le2$
\cite{2018-Vershynin}. In this case, we write
\[
\|X\|_{\psi_{\gamma}}=\inf\{t>0:\E\exp(|X|^{\gamma}/t^{\gamma})\le2\}.
\]
In the case $\gamma=2$, we say that $X$ is sub-Gaussian. We now
show that $\|\phi\|_{H^{\frac{1}{2}}(\bo)}$ is sub-Gaussian. 
\begin{lem}
If $\phi\sim\nu$, then the real random variable $\|\phi\|_{H^{\frac{1}{2}}(\bo)}$
is sub-Gaussian. In particular, there exist $c_{1},c_{2}>0$ depending
only on $\nu$ such that
\begin{equation}
\P(\|\phi\|_{H^{\frac{1}{2}}(\bo)}\ge t)\le2\exp(-c_{1}t^{2}),\qquad t\ge0,\label{eq:norm phi tail}
\end{equation}
and
\begin{equation}
\E\exp(\|\phi\|_{H^{\frac{1}{2}}(\bo)}^{2}/c_{2})\le2.\label{eq:norm phi sub}
\end{equation}
\end{lem}
\begin{proof}
Since $\phi$ is a sub-Gaussian random variable in $H^{1/2}(\bo)$,
we have that also $\|\phi\|_{H^{\frac{1}{2}}(\bo)}$ is sub-Gaussian
\cite{Koltchinskii-Lounici-2017,fukuda-1990}. The tail bound (\ref{eq:norm phi tail})
follows by \cite[Proposition~2.5.2]{2018-Vershynin}, and (\ref{eq:norm phi sub})
follows from the fact that $\left\Vert \|\phi\|_{H^{\frac{1}{2}}(\bo)}\right\Vert _{\psi_{2}}$
is finite by definition.
\end{proof}
Next, we show that the random variable $\zeta(u_{1},\dots,u_{n})(x)^{2}$
is $\frac{1}{n}$-subexponential.
\begin{lem}
\label{lem:subexp}Take $x\in\overline{\Omega'}$. The random variable
$\zeta(u_{1},\dots,u_{n})(x)^{2}$ is $\frac{1}{n}$-subexponential
and
\[
\|\zeta(u_{1},\dots,u_{n})(x)^{2}\|_{\psi_{\frac{1}{n}}}\le C,
\]
for some $C$ depending only on $\Omega$, $\Omega'$, $\Lambda$,
$\D$, $\zeta$ and $\nu$.
\end{lem}
\begin{proof}
By an abuse of notation, several positive constants depending only
on $\Omega$, $\Omega'$, $\Lambda$, $\D$, $\zeta$ and $\nu$ will
be denoted by the same letter $C>0$.

Using that $\zeta$ is bounded we obtain
\[
\|\zeta(u_{1},\dots,u_{n})\|_{C^{0,1/2}(\overline{\Omega'})}\le C\prod_{i=1}^{n}\|u_{i}\|_{C^{1,1/2}(\overline{\Omega'})}.
\]
Thus, classical elliptic regularity \cite[Theorem 8.32]{GILBARG-2001}
yields
\begin{equation}
|\zeta(u_{1},\dots,u_{n})(x)|\le\|\zeta(u_{1},\dots,u_{n})\|_{C^{0,1/2}(\overline{\Omega'})}\le C\prod_{i=1}^{n}\|\phi_{i}\|_{H^{\frac{1}{2}}(\bo)}.\label{eq:holder estimate}
\end{equation}

Take now $K>0$ to be chosen later. We now argue as in \cite[Lemma~2.7.7]{2018-Vershynin}.
By using (\ref{eq:holder estimate}) and the inequality of arithmetic
and geometric means we readily derive
\[
\begin{split}\exp\left(\frac{|\zeta(u_{1},\dots,u_{n})(x)|^{2/n}}{K}\right) & \le\exp\left(\frac{C}{K}\left(\prod_{i=1}^{n}\|\phi_{i}\|_{H^{\frac{1}{2}}(\bo)}\right)^{2/n}\right)\\
 & \le\exp\left(\frac{C}{K}\left(\frac{\sum_{i=1}^{n}\|\phi_{i}\|_{H^{\frac{1}{2}}(\bo)}}{n}\right)^{2}\right)\\
 & \le\exp\left(\frac{C}{Kn}\sum_{i=1}^{n}\|\phi_{i}\|_{H^{\frac{1}{2}}(\bo)}^{2}\right)\\
 & =\prod_{i=1}^{n}\exp\left(\frac{C}{Kn}\|\phi_{i}\|_{H^{\frac{1}{2}}(\bo)}^{2}\right).
\end{split}
\]
Using again the inequality of arithmetic and geometric means and Hölder
inequality in $\R^{n}$ we obtain
\[
\begin{split}\exp\left(\frac{|\zeta(u_{1},\dots,u_{n})(x)|^{2/n}}{K}\right) & \le\left(\frac{1}{n}\sum_{i=1}^{n}\exp\left(\frac{C}{Kn}\|\phi_{i}\|_{H^{\frac{1}{2}}(\bo)}^{2}\right)\right)^{n}\\
 & \le\frac{n^{n-1}}{n^{n}}\sum_{i=1}^{n}\exp\left(\frac{C}{K}\|\phi_{i}\|_{H^{\frac{1}{2}}(\bo)}^{2}\right)\\
 & \le\frac{1}{n}\sum_{i=1}^{n}\exp\left(\frac{C}{K}\|\phi_{i}\|_{H^{\frac{1}{2}}(\bo)}^{2}\right).
\end{split}
\]
Therefore, in view of (\ref{eq:norm phi sub}), we have
\[
\E\exp\left(\frac{|\zeta(u_{1},\dots,u_{n})(x)|^{2/n}}{c_{2}C}\right)\le2,
\]
and the result follows.
\end{proof}
We recall the following concentration inequality for $\gamma$-subexponential
random variables, which immediately follows from \cite[Theorem~1.3]{2021-gotze-et-al}
(see also \cite{Hitczenko-etal-1997}).
\begin{lem}
\label{lem:concentration}Let $X_{1},\dots,X_{N}$ be i.i.d.\ $\gamma$-subexponential
real random variables such that $\E X_{l}=\mu$ and $\|X_{l}\|_{\psi_{\gamma}}\le M$
for some $\gamma\in(0,1]$. Then
\[
\P\left(\left|\frac{1}{N}\sum_{l=1}^{N}X_{l}-\mu\right|\ge t\right)\le2\exp\left(-C_{\gamma}\min\left(Nt^{2}/M^{2},t^{\gamma}N^{\gamma}/M^{\gamma}\right)\right).
\]
\end{lem}
We now apply this result to the random variable $\zeta(u_{1},\dots,u_{n})(x)^{2}$:
as in the statement of Theorem~\ref{thm:constraints}, let $\phi_{i}^{l}\sim\nu$
be sampled i.i.d.\ and denote the corresponding solutions to (\ref{eq:ui dirichelt})
by $u_{i}^{l}$.
\begin{lem}
\label{lem:concentration application}There exist $C',C''>0$ depending
only on $\Omega$, $\Omega'$,  $\Lambda$, $\D$, $\zeta$ and $\nu$
such that
\[
\P\left(\frac{1}{N}\sum_{l=1}^{N}\zeta(u_{1}^{l},\dots,u_{n}^{l})(x)^{2}\le C''\right)\le2\exp\left(-C'N^{1/n}\right),\qquad x\in\overline{\Omega'}.
\]
\end{lem}
\begin{proof}
By an abuse of notation, several positive constants depending only
on $\Omega$, $\Omega'$, $\Lambda$, $\D$, $\zeta$ and $\nu$ will
be denoted by the same letter $C>0$.

Take $x\in\overline{\Omega'}$ and set $X_{l}=\zeta(u_{1}^{l},\dots,u_{n}^{l})(x)^{2}$.
By Lemma~\ref{lem:mean eta} we have $\mu:=\E X_{l}\ge\eta$. By
Lemma~\ref{lem:subexp}, $X_{l}$ is $\frac{1}{n}$-subexponential
and $\|X_{l}\|_{\psi_{\frac{1}{n}}}\le C$. We readily derive
\[
\begin{split}\P\left(\frac{1}{N}\sum_{l=1}^{N}X_{l}\le\frac{\eta}{2}\right) & \le\P\left(\frac{1}{N}\sum_{l=1}^{N}X_{l}\le\mu-\frac{\eta}{2}\right)\\
 & =\P\left(\mu-\frac{1}{N}\sum_{l=1}^{N}X_{l}\ge\frac{\eta}{2}\right)\\
 & \le\P\left(\left|\frac{1}{N}\sum_{l=1}^{N}X_{l}-\mu\right|\ge\frac{\eta}{2}\right).
\end{split}
\]
By Lemma~\ref{lem:concentration} we have
\[
\begin{split}\P\left(\left|\frac{1}{N}\sum_{l=1}^{N}X_{l}-\mu\right|\ge\frac{\eta}{2}\right) & \le2\exp\left(-C\min\left(N\eta^{2}/C^{2},\eta^{1/n}N^{1/n}/C^{1/n}\right)\right)\\
 & \le2\exp\left(-C\min\left(N\eta^{2},\eta^{1/n}N^{1/n}\right)\right)\\
 & \le2\exp\left(-CN^{1/n}\right).
\end{split}
\]
Setting $C''=\frac{\eta}{2}$, by these two estimates we obtain
\[
\P\left(\frac{1}{N}\sum_{l=1}^{N}X_{l}\le C''\right)\le2\exp\left(-CN^{1/n}\right),
\]
as desired.
\end{proof}
We are now ready to prove Theorem~\ref{thm:constraints}.
\begin{proof}[Proof of Theorem~\ref{thm:constraints}]
By an abuse of notation, several positive constants depending only
on $\Omega$, $\Omega'$,  $\Lambda$, $\D$, $\zeta$ and $\nu$
will be denoted by the same letter $C$. 

Take $t\ge0$ to be chosen later. By (\ref{eq:norm phi tail}) we
have that
\[
\P(\|\phi_{i}^{l}\|_{H^{\frac{1}{2}}(\bo)}\ge t)\le2\exp(-Ct^{2}),\qquad i=1,\dots,n,\;l=1,\dots,N.
\]
Thus, the union bound yields
\[
\P\left(\max_{i,l}\|\phi_{i}^{l}\|_{H^{\frac{1}{2}}(\bo)}\ge t\right)\le2nN\exp(-Ct^{2}).
\]
In other words, with probability greater than or equal to $1-2nN\exp(-Ct^{2}),$
we have $\max_{i,l}\|\phi_{i}^{l}\|_{H^{\frac{1}{2}}(\bo)}\le t$.
By (\ref{eq:holder estimate}) we have
\[
\|\zeta(u_{1}^{l},\dots,u_{n}^{l})\|_{C^{0,1/2}(\overline{\Omega'})}\le C\prod_{i=1}^{n}\|\phi_{i}\|_{H^{\frac{1}{2}}(\bo)}\le Ct^{n}.
\]
In other words, using for simplicity the equivalent $\sup$ norm in
$\R^{d}$, we have
\begin{equation}
|\zeta(u_{1}^{l},\dots,u_{n}^{l})(x)-\zeta(u_{1}^{l},\dots,u_{n}^{l})(y)|\le Ct^{n}\|x-y\|_{\infty}^{1/2},\qquad x,y\in\overline{\Omega'}.\label{eq:zeta holder}
\end{equation}
Let $C''>0$ be given by Lemma~\ref{lem:concentration application}
and set $r=\frac{C''}{\left(2Ct^{n}\right)^{2}}$. It is possible
to cover $\overline{\Omega'}$ with $M\le Cr^{-d}$ balls (with respect
to the norm $\|\cdot\|_{\infty}$):
\begin{equation}
\overline{\Omega'}\subseteq\bigcup_{j=1}^{M}B_{\infty}(x_{j},r),\label{eq:cover}
\end{equation}
where $x_{1},\dots,x_{M}\in\overline{\Omega'}$. In view of (\ref{eq:zeta holder}),
for $j=1,\dots,M$ and $x\in B_{\infty}(x_{j},r)\cap\overline{\Omega'}$
we have
\[
|\zeta(u_{1}^{l},\dots,u_{n}^{l})(x)-\zeta(u_{1}^{l},\dots,u_{n}^{l})(x_{j})|\le Ct^{n}\|x-x_{j}\|_{\infty}^{1/2}\le Ct^{n}r^{1/2}=\sqrt{C''}/2,
\]
so that
\begin{equation}
\begin{split}|\zeta(u_{1}^{l},\dots,u_{n}^{l})(x)| & \ge|\zeta(u_{1}^{l},\dots,u_{n}^{l})(x_{j})|-|\zeta(u_{1}^{l},\dots,u_{n}^{l})(x)-\zeta(u_{1}^{l},\dots,u_{n}^{l})(x_{j})|\\
 & \ge|\zeta(u_{1}^{l},\dots,u_{n}^{l})(x_{j})|-\sqrt{C''}/2.
\end{split}
\label{eq:reverse-holder}
\end{equation}

Next, by Lemma~\ref{lem:concentration application}, 
\[
\P\left(\frac{1}{N}\sum_{l=1}^{N}\zeta(u_{1}^{l},\dots,u_{n}^{l})(x_{j})^{2}\le C''\right)\le2\exp\left(-C'N^{1/n}\right),\qquad j=1,\dots,M.
\]
Thanks to the union bound, we have that
\[
\frac{1}{N}\sum_{l=1}^{N}\zeta(u_{1}^{l},\dots,u_{n}^{l})(x_{j})^{2}\ge C'',\qquad j=1,\dots,M,
\]
with probability greater than or equal to $1-2M\exp\left(-C'N^{1/n}\right)$.
Note that the previous condition implies that for every $j=1,\dots,M$
there exists $l=1,\dots,N$ such that $|\zeta(u_{1}^{l},\dots,u_{n}^{l})(x_{j})|\ge\sqrt{C''}$.
Thus, in view of (\ref{eq:cover}) and (\ref{eq:reverse-holder})
we have that for every $x\in\overline{\Omega'}$ there exists $l=1,\dots,N$
such that $|\zeta(u_{1}^{l},\dots,u_{n}^{l})(x)|\ge\sqrt{C''}/2$.
In other words, we proved (\ref{eq:max}) with $\Cppp=\sqrt{C''}/2.$
This happens with probability greater than or equal to
\[
1-2M\exp\left(-C'N^{1/n}\right)-2nN\exp(-Ct^{2}).
\]
Now, choose $t=N^{\frac{1}{2n}},$ so that $r=\frac{C}{N}$ and $M\le Cr^{-d}\le CN^{d}$,
so that
\[
\begin{split}1-2M\exp\left(-C'N^{1/n}\right) & -2nN\exp(-Ct^{2})\\
 & \ge1-CN^{d}\exp\left(-C'N^{1/n}\right)-2nN\exp(-CN^{1/n})\\
 & \ge1-\Cp N^{d}\exp\left(-\Cpp N^{1/n}\right),
\end{split}
\]
where we have used that $N\ge n^{\frac{1}{d-1}}$. This concludes
the proof.
\end{proof}

\section{\label{sec:Proof-of-Runge}Proof of Theorem~\ref{thm:quantitative runge}}

We need three technical lemmata.
\begin{lem}
\label{lem:weak*}Let $\Omega\subseteq\R^{d}$ be a bounded Lipschitz
domain, $p\in(1,+\infty)$ and consider the canonical embedding $I\colon L^{\infty}(\Omega)\to W^{-1,p}(\Omega):=\bigl(W_{0}^{1,p'}(\Omega)\bigr)'$
defined by
\[
\langle Iq,v\rangle_{W^{-1,p}(\Omega)\times W_{0}^{1,p'}(\Omega)}=\int_{\Omega}qv\,dx,\qquad q\in L^{\infty}(\Omega),v\in W_{0}^{1,p'}(\Omega).
\]
If $q_{n},q\in L^{\infty}(\Omega)$ and $q_{n}\to q$ weak{*} in $L^{\infty}(\Omega),$
then $Iq_{n}\to Iq$ in $W^{-1,p}(\Omega)$. 
\end{lem}
\begin{proof}
Let $q_{n},q\in L^{\infty}(\Omega)$ be such that $q_{n}\to q$ weak{*}
in $L^{\infty}(\Omega),$ namely,
\[
\int_{\Omega}q_{n}v\,dx\to\int_{\Omega}qv\,dx,\qquad v\in L^{1}(\Omega).
\]
In particular, since $\Omega$ is bounded, we have $L^{p'}(\Omega)\subseteq L^{1}(\Omega)$,
so that
\[
\int_{\Omega}q_{n}v\,dx\to\int_{\Omega}qv\,dx,\qquad v\in L^{p'}(\Omega).
\]
In other words, $q_{n}\to q$ weakly in $L^{p}(\Omega)$. Let $\iota\colon W_{0}^{1,p'}(\Omega)\to L^{p'}(\Omega)$
be the canonical embedding, which is compact thanks to the Rellich--Kondrachov
theorem, and
\[
T\colon L^{p}(\Omega)\to\bigl(L^{p'}(\Omega)\bigr)',\qquad\langle Tf,g\rangle_{L^{p'}(\Omega)'\times L^{p'}(\Omega)}=\int_{\Omega}fg\,dx,
\]
be the canonical isomorphism. Since $\iota^{*}\colon\bigl(L^{p'}(\Omega)\bigr)'\to W^{-1,p}(\Omega)$
is compact, we have that $\iota^{*}T\colon L^{p}(\Omega)\to W^{-1,p}(\Omega)$
is compact too. As a consequence, using that $q_{n}\to q$ weakly
in $L^{p}(\Omega)$, we have $\iota^{*}Tq_{n}\to\iota^{*}Tq$ in $W^{-1,p}(\Omega)$.
It remains to show that $\iota^{*}T|_{L^{\infty}(\Omega)}=I$, namely,
that it is the canonical embedding. For every $u\in L^{\infty}(\Omega)$
and $v\in W_{0}^{1,p'}(\Omega)$ we have
\[
\begin{split}\langle\iota^{*}Tu,v\rangle_{W^{-1,p}(\Omega)\times W_{0}^{1,p'}(\Omega)} & =\langle Tu,\iota v\rangle_{L^{p'}(\Omega)'\times L^{p'}(\Omega)}\\
 & =\int_{\Omega}u\iota(v)\,dx\\
 & =\int_{\Omega}uv\,dx\\
 & =\langle Iu,v\rangle_{W^{-1,p}(\Omega)\times W_{0}^{1,p'}(\Omega)},
\end{split}
\]
as desired.
\end{proof}
\begin{lem}
\label{lem:product-0}Let $\Omega\subseteq\R^{d}$ be a bounded Lipschitz
domain. The map 
\[
(u,v)\in H^{1}(\Omega)\times H_{0}^{1}(\Omega)\mapsto uv\in W_{0}^{1,(d+1)'}(\Omega)
\]
 is well-defined and bounded.
\end{lem}
\begin{proof}
The proof is split into three steps.

\emph{Step 1: the map is well defined in the case $d\ge3$.} Take
$(u,v)\in H^{1}(\Omega)\times H_{0}^{1}(\Omega)$. By Sobolev embedding,
$u,v\in L^{\frac{2d}{d-2}}(\Omega)$. By Cauchy-Schwartz inequality,
this implies that $uv\in L^{\frac{d}{d-2}}(\Omega)$. Thus, $uv\in L^{(d+1)'}(\Omega)$
since $(d+1)'=\frac{d+1}{d}\le\frac{d}{d-2}$. 

Next, we show that $\nabla(uv)\in L^{(d+1)'}(\Omega;\R^{d})$. We
have $\nabla(uv)=u\nabla v+v\nabla u$, with $u,v\in L^{\frac{2d}{d-2}}(\Omega)$
and \textbf{$\nabla u,\nabla v\in L^{2}(\Omega;\R^{d})$}. Hölder
inequality implies that $u\nabla v,v\nabla u\in L^{p}(\Omega;\R^{d})$
where
\[
\frac{1}{p}=\frac{d-2}{2d}+\frac{1}{2}=\frac{d-1}{d}=\frac{1}{d'}.
\]
Thus $u\nabla v,v\nabla u\in L^{d'}(\Omega;\R^{d})\subseteq L^{(d+1)'}(\Omega;\R^{d})$.

Finally, since $v|_{\bo}=0$, we have that $(uv)|_{\bo}=0$, so that
$uv\in W_{0}^{1,(d+1)'}(\Omega)$.

\emph{Step 2: the map is well defined in the case $d=2$. }Take $(u,v)\in H^{1}(\Omega)\times H_{0}^{1}(\Omega)$.
By Sobolev embedding, $uv\in L^{p}(\Omega)$ for every $p\in(1,+\infty)$,
and in particular $uv\in L^{(d+1)'}(\Omega)$. Next, we show that
$\nabla(uv)\in L^{(d+1)'}(\Omega;\R^{d})$. We have $\nabla(uv)=u\nabla v+v\nabla u$,
with $u,v\in L^{p}(\Omega)$ for every $p\in(1,+\infty)$ and \textbf{$\nabla u,\nabla v\in L^{2}(\Omega;\R^{d})$}.
Thus $u\nabla v,v\nabla u\in L^{\frac{3}{2}}(\Omega;\R^{d})$, as
desired. As above, $(uv)|_{\bo}=0$.

\emph{Step 3: the map is bounded. }The boundedness of the bilinear
map $(u,v)\mapsto uv$ follows immediately from the argument above,
by using the continuity of Sobolev embeddings.
\end{proof}
\begin{lem}
\label{lem:product}Let $\Omega\subseteq\R^{d}$ be a bounded Lipschitz
domain. The map $T\colon W^{-1,d+1}(\Omega)\times H^{1}(\Omega)\to H^{-1}(\Omega)$
defined by
\[
\langle T(q,u),v\rangle_{H^{-1}(\Omega)\times H_{0}^{1}(\Omega)}=\langle q,uv\rangle_{W^{-1,d+1}(\Omega)\times W_{0}^{1,(d+1)'}(\Omega)},
\]
for $q\in W^{-1,d}(\Omega)$, $u\in H^{1}(\Omega)$ and $v\in H_{0}^{1}(\Omega)$,
is well-defined and bounded.
\end{lem}
\begin{proof}
The map $T$ is well defined, because $uv\in W_{0}^{1,(d+1)'}(\Omega)$,
thanks to Lemma~\ref{lem:product-0}. We now show that $T$ is bounded.
By using again Lemma~\ref{lem:product-0} we readily obtain
\[
\begin{split}\|T(q,u)\|_{H^{-1}(\Omega)} & =\sup_{\|v\|_{H_{0}^{1}(\Omega)=1}}|\langle q,uv\rangle_{W^{-1,d+1}(\Omega)\times W_{0}^{1,(d+1)'}(\Omega)}|\\
 & \le\sup_{\|v\|_{H_{0}^{1}(\Omega)=1}}\|q\|_{W^{-1,d+1}(\Omega)}\|uv\|_{W_{0}^{1,(d+1)'}(\Omega)}\\
 & \le C\|q\|_{W^{-1,d+1}(\Omega)}\sup_{\|v\|_{H_{0}^{1}(\Omega)=1}}\|u\|_{H^{1}(\Omega)}\|v\|_{H_{0}^{1}(\Omega)}\\
 & =C\|q\|_{W^{-1,d+1}(\Omega)}\|u\|_{H^{1}(\Omega)},
\end{split}
\]
as desired.
\end{proof}
We are now ready to prove Theorem~\ref{thm:quantitative runge}.
\begin{proof}[Proof of Theorem \ref{thm:quantitative runge}]
The proof is split into several steps. For $\Omega'\subseteq\Omega$,
we use the notation $H_{L}^{1}(\Omega')=\{u\in H^{1}(\Omega'):Lu=0\text{ in }\Omega'\}$.
Note that $H_{L}^{1}(\Omega')$ is a vector subspace of $H^{1}(\Omega')$
because $L$ is linear, namely, (\ref{eq:PDE}) is a linear PDE.

\emph{Step 1: without loss of generality we can assume $\|h\|_{H^{1}(D)}=1$.}
Suppose that the result is true if $\|h\|_{H^{1}(D)}=1$. Observe
that, since $f$ is positively homogeneous, (\ref{eq:thesis}) is
equivalent to 
\begin{equation}
\left\Vert \frac{h}{\|h\|_{H^{1}(D)}}-\frac{u|_{D}}{\|h\|_{H^{1}(D)}}\right\Vert _{L^{2}(D)}\le\epsilon,\qquad f\left(\frac{u}{\|h\|_{H^{1}(D)}}\right)\le C.\label{eq:thesis-1}
\end{equation}
Since the PDE under consideration is linear, $\frac{h}{\|h\|_{H^{1}(D)}}$
is still a local solution in $D$, and so there exists $\tilde{u}\in H_{L}^{1}(\Omega)$
such that
\[
\|\frac{h}{\|h\|_{H^{1}(D)}}-\tilde{u}|_{D}\|_{L^{2}(D)}\le\epsilon,\qquad f(\tilde{u})\le C.
\]
Then $u=\|h\|_{H^{1}(D)}\tilde{u}\in H_{L}^{1}(\Omega)$ and satisfies
(\ref{eq:thesis-1}), as desired.\medskip{}

\emph{Step 2: a proof by contradiction.} By contradiction, assume
that such constant $C>0$ does not exist. Thus, for every $n\in\N$
there exist $a_{n}\in W^{1,\infty}\left(\Omega;\mathbb{R}^{d\times d}\right)$
and $q_{n}\in L^{\infty}(\Omega;\R)$ satisfying $a_{n}^{T}=a_{n}$
and (\ref{eq:coercive}), (\ref{eq:a-Lambda}), (\ref{eq:q-Lambda})
and (\ref{eq:0-eig}), $D_{n}$ as in the statement and $h_{n}\in H_{L_{n}}^{1}(D_{n})$,
where $L_{n}u=-\div(a_{n}\nabla u)+q_{n}u$, satisfying $\|h_{n}\|_{H^{1}(D_{n})}=1$
such that
\begin{equation}
\forall u\in H_{L_{n}}^{1}(\Omega)\quad\left(f(u|_{\bo})\le n\implies\|u|_{D_{n}}-h_{n}\|_{L^{2}(D_{n})}>\epsilon\right).\label{eq:contradiction}
\end{equation}
In the rest of the proof, we will consider several subsequences: in
order to simplify the exposition, with an abuse of notation, we will
never denote them.\medskip{}

\emph{Step 3: $q_{n}\to q$ weak{*} in $L^{\infty}(\Omega)$.} Viewing
$L^{\infty}(\Omega)$ as the dual of $L^{1}(\Omega)$, by the Banach--Alaoglu
theorem the closed ball $\overline{B_{L^{\infty}(\Omega)}(0,\Lambda)}$
is compact with respect to the weak{*} topology. Since $\overline{B_{L^{\infty}(\Omega)}(0,\Lambda)}$
is metrisable and $q_{n}\in\overline{B_{L^{\infty}(\Omega)}(0,\Lambda)}$
by (\ref{eq:q-Lambda}), there exist $q\in\overline{B_{L^{\infty}(\Omega)}(0,\Lambda)}$
and a subsequence of $(q_{n})$ such that $q_{n}\to q$ weak{*} in
$L^{\infty}(\Omega)$, namely,
\begin{equation}
\int_{\Omega}q_{n}v\,dx\to\int_{\Omega}qv\,dx,\qquad v\in L^{1}(\Omega).\label{eq:qn-to-q-wstar}
\end{equation}
(This is the so-called sequential Banach--Alaoglu theorem.)\medskip{}

\emph{Step 4: $a_{n}\to a$ in $L^{\infty}(\Omega;\R^{d\times d})$.}
Recall that $a_{n}\in W^{1,\infty}(\Omega;\R^{d\times d})$ and, by
(\ref{eq:a-Lambda}), $\|a_{n}\|_{W^{1,\infty}(\Omega)}\le\Lambda$
for every $n$. By the Ascoli--Arzelà theorem, there exists $a\in W^{1,\infty}(\Omega;\R^{d\times d})$
such that $\|a\|_{W^{1,\infty}(\Omega)}\le\Lambda$ and $a_{n}\to a$
in $L^{\infty}(\Omega;\R^{d\times d})$. In particular, since $a_{n}$
satisfies $a_{n}^{T}=a_{n}$ and (\ref{eq:coercive}) for every $n\in\N$,
then $a^{T}=a$ and $a$ satisfies (\ref{eq:coercive}) too.

Now that we have found limit points for $(q_{n})$ and $(a_{n})$,
let us denote $Lu:=-\div(a\nabla u)+qu$.\medskip{}

\emph{Step 5: $0$ is not a Dirichlet eigenvalue of $L$ in $\Omega$
and $L$ satisfies (\ref{eq:0-eig}).} Let $F\in H^{-1}(\Omega)$,
and let $u\in H_{0}^{1}(\Omega)$ be a solution to
\[
-\div(a\nabla u)+qu=F.
\]
This problem may be rewritten as
\[
-\div(a_{n}\nabla u)+q_{n}u=\div((a-a_{n})\nabla u)+(q_{n}-q)u+F.
\]
Since $L_{n}$ satisfies (\ref{eq:0-eig}), by Lemma~\ref{lem:product}
we have
\[
\begin{split}\|u & \|_{H^{1}(\Omega)}\le\Lambda\|\div((a-a_{n})\nabla u)+(q_{n}-q)u+F\|_{H^{-1}(\Omega)}\\
 & \le\Lambda\left(\|(a-a_{n})\nabla u\|_{L^{2}(\Omega)}+\|(q_{n}-q)u\|_{H^{-1}(\Omega)}+\|F\|_{H^{-1}(\Omega)}\right)\\
 & \le\Lambda\left(\|a-a_{n}\|_{L^{\infty}(\Omega;\R^{d\times d})}\|\nabla u\|_{L^{2}(\Omega)}+\|I(q_{n}-q)\|_{W^{-1,d+1}(\Omega)}\|u\|_{H^{1}(\Omega)}+\|F\|_{H^{-1}(\Omega)}\right),
\end{split}
\]
where $\|a\|_{L^{\infty}(\Omega;\R^{d\times d})}=\supess_{x\in\Omega}\|a(x)\|_{2}$,
$\|A\|_{2}$ denotes the operator norm of the matrix $A\in\R^{d\times d}$
and $I\colon L^{\infty}(\Omega)\to W^{-1,d+1}(\Omega)$ is the canonical
embedding (see Lemma~\ref{lem:weak*}). Note that $\|a-a_{n}\|_{L^{\infty}(\Omega;\R^{d\times d})}\to0$
by Step 4. Furthermore, $q_{n}\to q$ weak{*} in $L^{\infty}(\Omega)$
by Step 3, so that $\|I(q_{n}-q)\|_{W^{-1,d+1}(\Omega)}\to0$ by Lemma~\ref{lem:weak*}
with $p=d+1$. Altogether, we obtain
\[
\|u\|_{H^{1}(\Omega)}\le\Lambda\|F\|_{H^{-1}(\Omega)}.
\]
This shows that, if $F=0$, then $u=0$, so that $0$ is not a Dirichlet
eigenvalue of $L$. Furthermore, $L$ satisfies (\ref{eq:0-eig}).\medskip{}

\emph{Step 6: $D_{n}\to D$ with respect to the Hausdorff distance
and the volume distance.} By Blaschke's selection theorem \cite{1940-price},
there exists a convex domain $D\subseteq\Omega$ such that
\begin{equation}
d_{H}(D_{n},D)\to0,\label{eq:Hausdorff D_n to D}
\end{equation}
where $d_{H}$ denotes the Hausdorff distance and is defined by
\[
d(X,Y)=\max\left(\sup_{x\in X}d(x,Y),\sup_{y\in Y}d(y,X)\right).
\]
For later use, we also point out that in this setting the convergence
with respect to the Hausdorff distance is equivalent to the convergence
in the volume distance, namely, 
\begin{equation}
|(D_{n}\setminus D)\cup(D\setminus D_{n})|\to0,\label{eq:volume distance}
\end{equation}
see \cite{shephard-webster-1965}, where $|\cdot|$ denotes the Lebesgue
measure. In particular, by (\ref{eq:hypotheses D}) we have that $|D|\ge\Lambda^{-1}$
and that $D\Subset\Omega$.\medskip{}

\emph{Step 7: $\tilde{h}_{n}\to h$ weakly in $H^{1}(\Omega)$ with
$h|_{D}\in H_{L}^{1}(D)$.} In view of the classical extension theorem
for Sobolev spaces (see, e.g., \cite[Chapter VI,  Section 3]{stein-1970}
and \cite[Lemma 6.4]{burenkov-1999}), we can extend $h_{n}\in H^{1}(D_{n})$
to $\tilde{h}_{n}\in H^{1}(\Omega)$. Since $\|h_{n}\|_{H^{1}(D)}\le1$
for every $n\in\N$, by (\ref{eq:hypotheses D}) we have $\|\tilde{h}_{n}\|_{H^{1}(\Omega)}\le E$,
for some $E>0$ independent of $n$. Therefore, using again the Banach--Alaoglu
theorem, we obtain that there exists $h\in H^{1}(\Omega)$ such that
\begin{equation}
\tilde{h}_{n}\to h\quad\text{weakly in}\quad H^{1}(\Omega).\label{eq:hn-to-h-H1}
\end{equation}
By the Rellich--Kondrachov theorem, 
\begin{equation}
\tilde{h}_{n}\to h\quad\text{in}\quad L^{2}(\Omega).\label{eq:hn-to-h-L2}
\end{equation}
We now claim that $h|_{D}\in H_{L}^{1}(D).$ Recall that $h_{n}\in H_{L_{n}}^{1}(D_{n})$
for every $n\in\N,$ namely,
\begin{equation}
\int_{D_{n}}a_{n}\nabla h_{n}\cdot\nabla v\,dx+\int_{D_{n}}q_{n}h_{n}v\,dx=0,\qquad v\in C_{c}^{\infty}(D_{n}).\label{eq:weak-hn}
\end{equation}
We need to show that
\begin{equation}
\int_{D}a\nabla h\cdot\nabla v\,dx+\int_{D}qhv\,dx=0,\qquad v\in C_{c}^{\infty}(D).\label{eq:weak-h}
\end{equation}
Take $v\in C_{c}^{\infty}(D)$. Thus, by (\ref{eq:Hausdorff D_n to D}),
using that $D_{n}$ and $D$ are Lipschitz domains, in view of \cite[eqn. (1.5) and Lemma 7.4]{savare-schimperna-2002}
there exists $n_{0}\in\N$ such that $\supp v\subseteq D_{n}$ for
all $n\ge n_{0}$. Hence, for $n\ge n_{0}$ we have
\[
\begin{split} & \left|\int_{D_{n}}q_{n}h_{n}v\,dx-\int_{D}qhv\,dx\right|=\left|\int_{D}q_{n}\tilde{h}_{n}v\,dx-\int_{D}qhv\,dx\right|\\
 & \qquad\le\left|\int_{D}q_{n}\tilde{h}_{n}v\,dx-\int_{D}q_{n}hv\,dx\right|+\left|\int_{D}q_{n}hv\,dx-\int_{D}qhv\,dx\right|\\
 & \qquad\le\int_{D}|\tilde{h}_{n}-h||q_{n}v|\,dx+\left|\int_{D}(q_{n}-q)hv\,dx\right|\\
 & \qquad\le\|q_{n}\|_{L^{\infty}(D)}\|\tilde{h}_{n}-h\|_{L^{2}(D)}\|v\|_{L^{2}(D)}+\left|\int_{\Omega}(q_{n}-q)h\tilde{v}\,dx\right|,
\end{split}
\]
where $\tilde{v}\in L^{2}(\Omega)$ is the extension by zero of $v$.
The first of these two factors goes to zero thanks to (\ref{eq:hn-to-h-L2})
and (\ref{eq:q-Lambda}), while the second factor goes to zero by
(\ref{eq:qn-to-q-wstar}), since $h\tilde{v}\in L^{1}(\Omega)$. Therefore
\begin{equation}
\left|\int_{D_{n}}q_{n}h_{n}v\,dx-\int_{D}qhv\,dx\right|\to0\quad\text{as}\;n\to+\infty.\label{eq:limit1}
\end{equation}

Next, we consider the leading order term. Since the matrices $a$
and $a_{n}$ are symmetric, for $n\ge n_{0}$ we have
\[
\begin{split} & \left|\int_{D_{n}}a_{n}\nabla h_{n}\cdot\nabla v\,dx-\int_{D}a\nabla h\cdot\nabla v\,dx\right|\\
 & =\left|\int_{D}a_{n}\nabla\tilde{h}_{n}\cdot\nabla v\,dx-\int_{D}a\nabla h\cdot\nabla v\,dx\right|\\
 & \le\left|\int_{D}a_{n}\nabla\tilde{h}_{n}\cdot\nabla v\,dx-\int_{D}a\nabla\tilde{h}_{n}\cdot\nabla v\,dx\right|+\left|\int_{D}a\nabla\tilde{h}_{n}\cdot\nabla v\,dx-\int_{D}a\nabla h\cdot\nabla v\,dx\right|\\
 & \le\int_{D}\left|(a_{n}-a)\nabla\tilde{h}_{n}\cdot\nabla v\right|\,dx+\left|\int_{\Omega}(\nabla\tilde{h}_{n}-\nabla h)\cdot a\nabla\tilde{v}\,dx\right|\\
 & \le\int_{D}\|a_{n}(x)-a(x)\|_{2}|\nabla\tilde{h}_{n}(x)|\cdot|\nabla v(x)|\,dx+\left|\langle\nabla\tilde{h}_{n}-\nabla h,a\nabla\tilde{v}\rangle_{L^{2}(\Omega)}\right|\\
 & \le\|a_{n}-a\|_{L^{\infty}(D;\R^{d\times d})}\|\tilde{h}_{n}\|_{H^{1}(\Omega)}\|v\|_{H^{1}(D)}+\left|\langle\nabla(\tilde{h}_{n}-h),a\nabla\tilde{v}\rangle_{L^{2}(\Omega)}\right|.
\end{split}
\]
Using that $\|\tilde{h}_{n}\|_{H^{1}(\Omega)}\le E$ for every $n$
and that $a_{n}\to a$ in $L^{\infty}(\Omega;\R^{d\times d})$, we
obtain that the first factor goes to $0$ as $n\to+\infty$. The second
term goes to zero thanks to (\ref{eq:hn-to-h-H1}). Therefore
\begin{equation}
\left|\int_{D_{n}}a_{n}\nabla h_{n}\cdot\nabla v\,dx-\int_{D}a\nabla h\cdot\nabla v\,dx\right|\to0\quad\text{as}\;n\to+\infty.\label{eq:limit2}
\end{equation}
Finally, (\ref{eq:weak-h}) follows directly by (\ref{eq:weak-hn}),
(\ref{eq:limit1}) and (\ref{eq:limit2}).\medskip{}

\emph{Step 8: there exists $u\in H_{L}^{1}(\Omega)$ such that $f(u|_{\bo})<+\infty$
and $\|u|_{D}-h|_{D}\|_{L^{2}(D)}\le\frac{\epsilon}{4}$.} Note that
$L$ and $D$ satisfy all the assumptions of Proposition~\ref{prop:runge},
because $\Omega\setminus\overline{D}$ is connected (since $D$ is
convex and $D\Subset\Omega$). Thus, there exists $u'\in H_{L}^{1}(\Omega)$
such that $\|u'|_{D}-h|_{D}\|_{L^{2}(D)}\le\epsilon/8$. By density
of $f^{-1}([0,+\infty))$ in $H^{1/2}(\bo)$, there exist $g_{k}\in H^{1/2}(\bo)$
such that $f(g_{k})<+\infty$ and $g_{k}\to u'|_{\bo}$ in $H^{1/2}(\bo)$.
Let $u_{k}\in H^{1}(\Omega)$ be the unique solution to
\[
\left\{ \begin{array}{ll}
Lu_{k}=0 & \text{in}\;\Omega,\\
u_{k}=g_{k} & \text{on}\;\bo.
\end{array}\right.
\]
Then $u_{k}\to u'$ in $H^{1}(\Omega)$ because $L^{-1}$ is continuous.
Take $k_{0}\in\N$ such that $\|u_{k_{0}}-u'\|_{H^{1}(\Omega)}\le\epsilon/8$,
so that
\[
\|u_{k_{0}}|_{D}-h|_{D}\|_{L^{2}(D)}\le\|u_{k_{0}}|_{D}-u'|_{D}\|_{L^{2}(D)}+\|u'|_{D}-h|_{D}\|_{L^{2}(D)}\le\epsilon/4.
\]
It is enough to set $u=u_{k_{0}}$.\medskip{}

\emph{Step 9: $\|h-u\|_{L^{2}(D_{n}\setminus D)}\to0$.} Set $w=h-u\in H^{1}(\Omega)$.
By Sobolev embedding, we have $w\in L^{q}(\Omega)$ for some $q>2$
(depending on $d$). By Hölder inequality, setting $p=q/2>1$, we
obtain
\[
\|w\|_{L^{2}(D_{n}\setminus D)}^{2}\le\|\chi_{D_{n}\setminus D}\|_{L^{p'}(\Omega)}\||w|^{2}\|_{L^{p}(\Omega)}=|D_{n}\setminus D|^{\frac{1}{p'}}\|w\|_{L^{q}(\Omega)}^{2},
\]
where $p'=\frac{p}{p-1}\in(1,\infty)$. The right hand side goes to
zero by (\ref{eq:volume distance}), whence $\|h-u\|_{L^{2}(D_{n}\setminus D)}\to0$.\medskip{}

\emph{Step 10: $u_{n}\to u$ in $H^{1}(\Omega)$, where $u_{n}\in H_{L_{n}}^{1}(\Omega)$.}
For $n\in\N$, let $u_{n}\in H^{1}(\Omega)$ be the unique solution
to the Dirichlet boundary value problem
\[
\left\{ \begin{array}{ll}
L_{n}u_{n}=0 & \text{in}\;\Omega,\\
u_{n}=u & \text{on}\;\bo.
\end{array}\right.
\]
Set $v_{n}=u_{n}-u\in H_{0}^{1}(\Omega)$. We have $L_{n}v_{n}=-L_{n}u=(L-L_{n})u$.
Arguing as in Step 5, we obtain that
\[
\|v_{n}\|_{H^{1}(\Omega)}\le\Lambda\left(\|a-a_{n}\|_{L^{\infty}(\Omega;\R^{d\times d})}\|\nabla u\|_{L^{2}(\Omega)}+\|I(q_{n}-q)\|_{W^{-1,d+1}(\Omega)}\|u\|_{H^{1}(\Omega)}\right),
\]
where the right hand side goes to $0$ as $n\to+\infty$, as desired.\medskip{}

\emph{Step 11: The contradiction.} Recall that, by (\ref{eq:hn-to-h-L2}),
step 9 and step 10 we have $\|\tilde{h}_{n}-h\|_{L^{2}(\Omega)}\to0$,
$\|h-u\|_{L^{2}(D_{n}\setminus D)}\to0$ and $\|u-u_{n}\|_{H^{1}(\Omega)}\to0$.
Choose $\bar{n}\in\N$ such that $f(u|_{\bo})\le\bar{n}$ and
\[
\|\tilde{h}_{\bar{n}}-h\|_{L^{2}(\Omega)}\le\epsilon/4,\qquad\|h-u\|_{L^{2}(D_{\bar{n}}\setminus D)}\le\epsilon/4,\qquad\|u-u_{\bar{n}}\|_{H^{1}(\Omega)}\le\epsilon/4.
\]
Since $f(u_{\bar{n}}|_{\bo})=f(u|_{\bo})\le\bar{n}$, by (\ref{eq:contradiction})
we have $\|h_{\bar{n}}-u_{\bar{n}}|_{D_{\bar{n}}}\|_{L^{2}(D_{\bar{n}})}>\epsilon$.
By Step 8 we have
\[
\begin{split}\|h|_{D_{\bar{n}}}-u|_{D_{\bar{n}}}\|_{L^{2}(D_{\bar{n}})}^{2} & \le\|h|_{D}-u|_{D}\|_{L^{2}(D)}^{2}+\|h|_{D_{\bar{n}}\setminus D}-u|_{D_{\bar{n}}\setminus D}\|_{L^{2}(D_{\bar{n}}\setminus D)}^{2}\\
 & \le2(\epsilon/4)^{2}\\
 & \le(\epsilon/2)^{2},
\end{split}
\]
whence, by the triangle inequality, we obtain
\[
\|h_{\bar{n}}-u_{\bar{n}}|_{D_{\bar{n}}}\|_{L^{2}(D_{\bar{n}})}\le\|\tilde{h}_{\bar{n}}-h\|_{L^{2}(\Omega)}+\|h|_{D_{\bar{n}}}-u|_{D_{\bar{n}}}\|_{L^{2}(D_{\bar{n}})}+\|u-u_{\bar{n}}\|_{L^{2}(\Omega)}\le\epsilon,
\]
a contradiction. This concludes the proof.
\end{proof}
\bibliographystyle{abbrv}
\bibliography{book}

\end{document}